\documentclass[11pt,english]{article}
\usepackage{amsmath, amsthm, latexsym, amssymb,url,dsfont}
\usepackage{amsfonts}
\usepackage{epsfig}
\usepackage{enumerate}
\usepackage{graphicx}
\usepackage{yhmath}
\usepackage{mathdots}
\usepackage{mathtools} 
\usepackage{pifont}

\usepackage{tikz}
\usetikzlibrary{matrix,arrows,decorations.pathmorphing}
\usepackage{tikz-cd}

\usepackage{relsize}
\newcommand{\bigot}{\mathop{\large \mathlarger{\mathlarger{\bot}}}}

\newcommand{\shrinkmargins}[1]{
  \addtolength{\textheight}{#1\topmargin}
  \addtolength{\textheight}{#1\topmargin}
  \addtolength{\textwidth}{#1\oddsidemargin}
  \addtolength{\textwidth}{#1\evensidemargin}
  \addtolength{\topmargin}{-#1\topmargin}
  \addtolength{\oddsidemargin}{-#1\oddsidemargin}
 \addtolength{\evensidemargin}{-#1\evensidemargin}
  }

\shrinkmargins{.5}

\theoremstyle{plain}

\newtheorem{theorem}{Theorem}[section]
\newtheorem{corollary}[theorem]{Corollary}
\newtheorem{lemma}[theorem]{Lemma}
\newtheorem{proposition}[theorem]{Proposition}
\newtheorem{question}[theorem]{Question}

\newtheorem*{teo}{Theorem}
\newtheorem*{coro}{Corollary}

\newtheorem{definition}[theorem]{Definition}

\theoremstyle{remark}
\newtheorem{remark}[theorem]{Remark}

\theoremstyle{definition}
\newtheorem{example}[theorem]{Example}

\def \Z { \mathbb{Z}}
\def \Q { \mathbb{Q}}

\def \C { \mathbb{C}}

\def \R { \mathbb{R}}

\def \tr { {\rm tr}}

\begin{document}

\thispagestyle{empty}
\setcounter{tocdepth}{7}

\title{The integral trace form as a complete invariant for real  $S_n$ number fields.}
\author{Guillermo Mantilla-Soler \and Carlos Rivera}

\date{}

\maketitle

\begin{abstract}

In the past the first named author has studied to what extent the integral trace can characterize a number field beyond what the discriminant does. The cases of cyclic number fields and non-totally real fields are more or less settled, concluding that for such fields the integral trace does not always characterize the field. In this paper we show that the integral trace is a complete invariant for degree $n$, $S_n$ real number number fields that satisfy certain ramification bound. Among the real $S_n$ fields that our results cover, there are those of square free different ideal. Moreover, for such fields we find an explicit description of the isometry group of the integral trace.
\end{abstract}

\section{ Introduction}
 
In the first part of the introduction we briefly present some of the results of the paper. 
In the second part of the introduction we give a detailed description as to how the ideas, statements and definitions behind the main results came to be. 
 
\subsection{Main results}

Given a number field $K$ we denote by $(O_K,{\rm tr}_{K/\mathbb{Q}})$ the pair given by the free $\Z$-module $O_{K}$ together with the symmetric bilinear pairing given by the trace \begin{displaymath}
\begin{array}{cccc}
 \mathrm{tr}_{K/\Q}: & O_{K} \times O_{K} &\rightarrow& \Z  \\  & (x,y) & \mapsto &
\mathrm{Tr}_{K/\Q}(xy).
\end{array}
\end{displaymath}

Some of the results we proved in this paper are stated next.

\begin{teo}[cf.Theorem \ref{ElGeneralisimo}]
Let $K$ be a degree $n$ totally real $S_{n}$ number field and $L$ be any $S_n$ number field. Let $d_{s}(K,L)$ be the integer defined in \ref{LosEnterosQuePartenDisc}.
\begin{enumerate}[(i)]
    \item Suppose that $n \ge d_{s}(K,L)^{2}$. Then, \[(O_K,{\rm tr}_{K/\mathbb{Q}}) \cong  (O_L,{\rm tr}_{L/\mathbb{Q}}) \ \mbox{if and only if} \ K \cong L.\]
    
    \item If  $K$ does not have wild ramification, then the condition $n \ge d_{s}(K,L)^{2}$ can be improved to $n \ge {\rm rad}(d_{s}(K,L))^{2}$.
    
    \item  If  $K$ and $L$ do not have wild ramification except possibly for some primes with ramification index $2$ lying over $2$, then the condition $n \ge d_{s}(K,L)^{2}$ can also be improved to $n \ge {\rm rad}(d_{s}(K,L))^{2}$ .
\end{enumerate}
\end{teo}

\begin{definition}\label{LosEnterosQuePartenDisc}

Let $K$ be a number field and let $d_{K}$ be its discriminant. For $p$ a rational prime we let $e_{p}(K)$ be the maximal ramification index of a prime in $K$ lying over $p$. For a number field $L$ we define $d_{f}(K,L)$ as the product  $\prod p^{v_p(d_K)}$ where $p$ runs over the set of odd primes such that $e_p(K)= 2=e_p(L)$, and set  $d_{s}(K,L):=d_K/d_f(K,L) \in \Z$.

\end{definition}

\begin{remark}
Notice that if $K$ and $L$ have the same discriminant $d$, then $d_s(K,L)$ is bounded by the non-square free part of $d$, i.e., $\displaystyle d_s(K,L) \leq \prod_{p, v_p(d) >  1} p^{v_p(d)}$.
\end{remark}

 As a consequence of the above we have that the integral trace is a complete invariant for $S_n$ real fields with at worst quadratic ramification.

 \begin{teo}[cf.Theorem \ref{TotallyRealFundDiscG}]
Let $K,L$ be $S_n$ number fields. Suppose that $K$ is totally real and that the ramification index of every prime in $K$ and $L$ over $\Q$ is at most $2$. Then, $$(O_K,{\rm tr}_{K/\mathbb{Q}}) \cong  (O_L,{\rm tr}_{L/\mathbb{Q}}) $$ if and only if $K \cong L$.
\end{teo}
 
 \begin{coro}[Corollary \ref{IntTraceCompleteInv}]
 The integral trace form is a complete invariant for totally real number fields of fundamental discriminant.
 
 \end{coro}
 
 \begin{example}
 
 Let $K$ and $L$ be the sextic  fields defined respectively by the polynomials $x^6-3x^5-6x^4+17x^3+2x^2-8x+1$ and $x^6-x^5-12x^4-x^3+22x^2-1$, and let $d_{f}:= d_f(K,L)$ and $d_{s}:= d_s(K,L).$ The fields $K$ and $L$ are totally real of discriminant $d=2^3\cdot 11 \cdot 619 \cdot 11411$. In this case $d_{s}=8$ and $d_{f}=77697499$, and since $d$ is fundamental  ${\rm rad}(d_{s})^{2}=4$. The fields $K$ and $L$ are not isomorphic, for instance a calculation shows that the prime $p=11$ has three primes in $K$ lying over it, while it has two in $L$, hence their integral traces are not isometric. 
 \end{example}
  
  \begin{teo}[cf.Theorem \ref{SqFDiff}]
Let $K$ and $L$ be  totally real number fields with square free different ideal. Then any isometry $\phi: (K, {\rm tr}_{K/\mathbb{Q}}) \rightarrow (L, {\rm tr}_{L/\mathbb{Q}})$ such that $\phi(O_K)=O_L$ is equal to plus or minus an isomorphism of fields $K \cong L$.
 \end{teo}
  
\subsubsection{Some applications}

As a byproduct of the methods used in the proofs of the above results we can provide, for real number fields with square free different, an explicit description of the automorphism group of the integral trace.

\begin{teo}[cf.Corollary \ref{ElAuto}]
Let $K$ be a totally real number field with square free different ideal. Then, \[{\rm Aut} \left((O_K,{\rm tr}_{K/\mathbb{Q}}) \right) \cong 
\Z/2\Z \times {\rm Aut}(K).\] In particular, if $n>2$ and $K$ is a $S_{n}$ number field then the automorphism group of the integral trace is ``trivial".
\end{teo}

\begin{example}

Let $K$ be the biquadratic field $\Q(\sqrt{5}, \sqrt{13})$. This is a totally real number field with  different ideal the principal generated by $\sqrt{65}$. Since $\sqrt{5}$ and $\sqrt{13}$ generate prime ideals in $O_{K}$ the different ideal is square free. Since $K$ is a Galois number field, with Galois group the Klein group, then  ${\rm Aut} \left((O_K,{\rm tr}_{K/\mathbb{Q}}) \right) \cong 
\Z/2\Z \times \Z/2\Z \times \Z/2\Z.$ 

\end{example}

A consquence of the above result is that for real fields of square free discriminant the automorphism group of the integral trace is trivial.

\begin{coro}[Corollary \ref{AutTraceFundDisc}]
Let $K$ be a real number field of degree at least $3$ and with square free discriminant. Then, \[{\rm Aut} \left((O_K,{\rm tr}_{K/\mathbb{Q}}) \right) \cong 
\Z/2\Z.
\]
\end{coro}

 A further application of the results here, which is ongoing work, is a method to count real number fields with square free discriminant. Using the results of this paper, together with some control of the local behavior of the trace, we can apply Siegel's mass formula to obtain some counting results; at least for fields of small degree. 
 
\subsection{Motivation on definitions and results} 
 
One of the main problems in algebraic number theory is to give a satisfactory way of deciding whether or not two number fields are isomorphic. Inspired by the style of question of  Cornelissen and Marcolli in \cite{Corn1}, this query can be formulated as follows:
\begin{question}\label{LaPregunta}
Can one describe an isomorphism between number fields $K$, $L$ from an associated mathematical/arithmetical object? In other words, can we find a complete invariant for number fields?
\end{question} 
Several natural objects have been at the center of study of this question by several authors; e.g., the Dedekind zeta function, the ring of Adeles, the group of Dirichlet Characters, the absolute Galois group (see \cite{Corn1, Iwa, Komat, Neu, Perlis1, Uchida}). From all these the only one that is a complete invariant, thanks to a famous result of Neukirch and Uchida, see \cite{Neu, Uchida}, is the absolute Galois group.\\

All the above invariants have in common that they refine the discriminant; i.e., having the same invariant for $K$ and $L$ implies that the two fields have the same discriminant. Thanks to Hermite and Minkowski it is known that, up to conjugation, there are only finitely many number fields of a given discriminant. This result has led to the study of number fields via their discriminants, and specifically to the industry of results about the asymptotics of the number of fields up to a given discriminant bound, sometimes with specific signature types or Galois groups. Some of the examples of such investigations can be found in works of Bhargava, Cohen, Davenport-Heilbron,  Datskovsky-Wright, Ellenberg-Venkatesh, Jones, Kl\"uners, Lemke Oliver-Thorne, Malle, P. Harron, R. Harron, Roberts, Taniguchi-Thorne, Shnidman, Wang and many others (See \cite{bha2, BH, bhashn, DabWri, EV1, ev, Klu, Jones, LemkeThorne, PH, RH, Roberts, thorne, DTerr, wang}). Since the discriminant of a number field $K$ of degree $n$ is basically the signed co-volume of the lattice $O_{K}$ inside of $\R^{n}$ it is only natural to study the isometry class of such lattices as a way to refine the discriminant. In other words the isometry class of the integral bilinear pairing \begin{displaymath}
\begin{array}{cccc}
 \mathrm{tr}_{K/\Q}: & O_{K} \times O_{K} &\rightarrow& \Z  \\  & (x,y) & \mapsto &
\mathrm{Tr}_{K/\Q}(xy).
\end{array}
\end{displaymath}
 is the first natural object that comes to mind when dealing with the problem of dividing the classification of number fields of the same discriminant in a smaller subclass. This invariant, and the closely related {\it  shape}, are used for instance by Ellenberg and Venkatesh in their paper about asymptotics of number fields of bounded discriminant; see \cite[Remark 3.2]{ev}. \\ 
 
 Trace forms over fields, and their applications, have been extensively studied by many authors; see for instance the works of Eva Bayer-Fluckiger, J.P. Serre and their co-authors \cite{BaLe, BaPaSe, SerreTrace} and the work of Robert Perlis and his co-authors \cite{Perlis}. However, most of the literature around the classification results about the integral trace over number fields is scarce. These kinds of classification questions have been studied by the second named author from four different, basically disjoint, perspectives:
 
 \begin{itemize}
 
 \item [(i)] The first one uses the techniques developed by M. Bhargava in his thesis, specifically the higher composition laws in cubes and their relation to cubic fields. See \cite{Man} for details and for a proof of 
 
 \begin{theorem}\label{LoCubicos} Let $K,L$ be cubic fields such that $K$ has positive fundamental discriminant. Then,  \[(O_K,{\rm tr}_{K/\mathbb{Q}}) \cong  (O_L,{\rm tr}_{L/\mathbb{Q}}) \ \mbox{if and only if} \ K \cong L.\] 
 
 \end{theorem}
 
 \item[(ii)] A local-global point of view. Using the Adelization of the orthogonal group one can show that genus and the spinor genus of the integral trace form coincide. Together with the classification genus of the integral trace, this method permits  to deal with the problem in the case of non-real fields that have no wild ramification. See \cite{Man5} and \cite{Man7} for details and a proof of

  \begin{theorem} Let $K,L$ be two  number fields with the same discriminant $d$, same signature and without wild ramification. Suppose that  the fields are non-totally real and that they do not have wild primes. Then,  \[  (O_K,{\rm tr}_{K/\mathbb{Q}}) \cong  (O_L,{\rm tr}_{L/\mathbb{Q}})  \iff  \left( \frac{\alpha_{K}^{p}}{p}\right)=  \left( \frac{\alpha_{L}^{p}}{p}\right) \ \mbox{for all odd primes $p$ that divide $d$}.\]  For extra details of the construction of the elements $\alpha_{K}^{p}$ see also \cite{Man9}.
 
 \end{theorem}

 \item[(iii)] Using that rational quadratic forms can be seen as cohomology classes in H$^{1}(\Q, {\rm O}_{n})$ and that Artin representations can be seen as elements of H$^{1}(\Q, {\rm GL}_{n}(\C))$ one can study the relation between Dedekind zeta functions of number fields and their integral traces. See \cite{Man6} for details and for a proof of
 
   \begin{theorem} Let $K,L$ be two  number fields. Suppose that  $K$ is non-totally real and that no rational prime has wild ramification in it. Then,  \[ \zeta_{K}(s)=\zeta_{L}(s) \Longrightarrow (O_K,{\rm tr}_{K/\mathbb{Q}}) \cong  (O_L,{\rm tr}_{L/\mathbb{Q}})  .\] 
 
 \end{theorem}

 \item[(iv)] Imposing the Galois theoretic conditions, mainly prime power Galois cyclic extensions, a classification of the integral trace and shape can also be obtained for such fields. See \cite{Man8} and \cite{Bol} for details and a proof of
 
  \begin{theorem} Let $n$ be a positive integer Let $K,L$ be two Galois number fields with Galois group isomorphic to $\Z/n\Z$ and without wild primes. Then,  \[(O_K,{\rm tr}_{K/\mathbb{Q}}) \cong  (O_L,{\rm tr}_{L/\mathbb{Q}}) \ \mbox{if and only if} \ {\rm disc}(K) \cong {\rm disc}(L).\] 
 
 \end{theorem}

\end{itemize}
 
 \subsubsection{Motivation behind the definition of $d_{s}(K,L)$}
 
 Cases (i) and (iv) have an interesting duality for cubic fields. The former says that for square free discriminant the trace is a great improvement of the discriminant, while the latter says that for perfect square discriminants the trace does not provide new information. Hence, it is interesting to wonder what happens in the intermediate cases, i.e., neither square free nor perfect square discriminants. As the following example shows the answer could go either way
 
 \begin{example} Let $K_{1}, K_{2}$ and $K_{3}$ be the cubic fields defined respectively by the cubic polynomials $f_{1}:=x^3 - 36x - 78, f_{2}:=x^3 - 18x - 6$ and $f_{3} :=x^3 - 36x - 60$. A calculation shows that these fields are not isomorphic,  all have discriminant $2^2\cdot3^5\cdot23$ and \[(O_{K_{1}},{\rm tr}_{K_{1}/\mathbb{Q}}) \cong  (O_{K_{3}},{\rm tr}_{K_{3}/\mathbb{Q}}) \not\cong  (O_{K_{2}},{\rm tr}_{K_{2}/\mathbb{Q}}). \]
 
 \end{example}

This lead us to think that to generalize case (i) to higher degree number fields we  need to generalize the square free condition of the discriminant; this was in fact supported by computational evidence. Number fields with square free discriminant have square free different ideal. It turns out, as a simple calculation shows, that for cubic fields both conditions are equivalent. 

\begin{remark}
Let $K$ be a cubic number field. Then, $K$ has square free different ideal if and only if it has square free discriminant.
\end{remark}

A general equivalent condition to having square free different ideal is given by the following result. We call a number field $K$ tame if no prime has wild ramification in $K$.

\begin{proposition}
Let $K$ be a number field. The field $K$ has square free different ideal if and only if $K$ is tame and for every prime $p$ the maximal ramification degree is at most $2$. 
\end{proposition}

\begin{proof} The result follows from the formula for the exponents of primes appearing in the different ideal. More explicitly, recall that for a prime $\mathfrak{P}$ in $O_{K}$ lying over a prime $p$ with ramification degree $e$, the exponent of $\mathfrak{P}$ on the different ideal is at least $e-1$. Moreover, it is $e-1$ if and only if $p\nmid e$.
\end{proof}
A possible generalization to having  square free discriminant could be having symmetric group as the Galois group and having ramification exponents bounded by $2$. This is explained thanks to the above and since number fields with fundamental discriminant have Galois group the full symmetric group. In fact this is the motivation behind Theorem \ref{TotallyRealFundDiscG}. The main results of this paper, Theorem \ref{ElGeneralisimo}, pushes this generalization even further. As long as the part of the discriminant that comes from prime with ramification degrees bigger than $2$ is not so big, then the integral trace is a complete invariant.\\

 Besides the motivation to find the right generalization to Theorem \ref{LoCubicos}, in this paper we introduce and develop the  Casimir pairings. With it we introduce a fifth approach to the integral trace invariant problem. This new method, combined with several of the points of view described above, yields to most of the results of the paper.

\subsection{Outline of the paper}

In \S \ref{CasimirPairings} we introduce the theory of Casimir Pairings and show many of its properties which we will need later on. In \S \ref{Disjoint} we study linearly disjoint number fields, the relation of Casimir Pairings with them and we give a general criterion to decide whether or not two number fields are conjugate. In \S \ref{CasimiroYLosNumberFields} we do a local-global study of Casimir elements associated to trace pairings over number fields; this section is the connection between the general theory of Casimir pairings and number fields. Finally in \S \ref{LosResultados} we provide the proofs to the number theoretic applications from all the results of the previous sections.

\section{Casimir Pairings}\label{CasimirPairings}

Let $V$ be a finite dimensional vector space over a field $k$ endowed with a nondegenerate bilinear form $B$, and let $V^*={\rm Hom}(V,k)$ be its dual. The isomorphism $ \Gamma_B:V \rightarrow V^*,\, v\mapsto \Gamma_B(v):=B(v,\cdot)$  determines a bilinear form on $V^*$, namely the pullback of $B$ via $\Gamma_B^{-1}$; we will denote this form by  $\langle \cdot , \cdot\rangle_{B}$ and we call it {\it Casimir pairing} associated to $B$.

For a field $k$ let us denote by ${\rm FVectB}_k$ the category of pairs $(V,B)$ where $V$ is a finite dimensional $k$-vector space and $B$ a nondegenerate $k$-bilinear form on $V$ with morphisms given by bijective $k$-linear isometries. 

\begin{proposition} The functor
\begin{center}
\begin{tikzcd}[column sep=large, row sep=tiny]
{\rm FVectB}_k \arrow[r, "\mathcal{CS}"] &  {\rm FVectB}_k \\
(V,B_V) \arrow[d, "\phi"'{name=f}] & (V^*, \langle \cdot, \cdot \rangle_{B_V})  \\[5ex]
(W,B_W) & (W^*, \langle \cdot, \cdot \rangle_{B_W}) \arrow[u, "\phi^*"'{name=Ff}]
\arrow[mapsto,from=f, to=Ff,shorten=3.7em]
\end{tikzcd}
\end{center}

\noindent mapping $(V,B)$ to $(V^*,\langle \cdot , \cdot\rangle_{B})$ and an isometry $\phi:V \rightarrow W$ to its dual $\phi^*$ is a well-defined duality of categories.
\end{proposition}

\begin{proof}
An object $(V,B)$ in ${\rm FVectB}_k$ comes equipped with the $k$-linear isomorphism $ \Gamma_B$, which we are forcing to be an isometry between $(V,B)$ and  $\mathcal{CS}(V,B)$. In particular, $\mathcal{CS}(V,B)$ is nondegenerate, i.e., the functor is well defined on objects. To check that it is well defined on morphisms, note that an isometry $\phi: (V,B) \rightarrow (W,B')$, by definition, makes the following diagram of $k$-linear maps commutative: \begin{center}
  \begin{tikzpicture}
\matrix(m)[matrix of math nodes,
row sep=2.6em, column sep=4.5em,
text height=1.5ex, text depth=0.25ex]
{V&V^*\\
W& W^*\\};
\path[->,font=\scriptsize,>=angle 90]
(m-1-1) edge node[auto] {$\Gamma_B$} (m-1-2)
edge node[auto] { $\phi$} (m-2-1)
(m-2-2) edge node[auto] {$\phi^*$} (m-1-2)
(m-2-1) edge node[auto] {$\Gamma_{B'}$} (m-2-2);
\end{tikzpicture}.  
\end{center}
 Thus if $\phi$ is bijective, the map $\phi^*: \mathcal{CS}(V,B) \rightarrow \mathcal{CS}(W,B')$ is also a bijective isometry.  An inverse for $\mathcal{CS}$ is given by the  functor  $\mathcal{CS}^t$ mapping $(V,B)$ into $V^*$ together with bilinear map that makes $\Gamma_B^{t}:V\rightarrow V^*, \Gamma_B^{t}(v):=B(\cdot,v)$ into an isometry. 
\end{proof}
We stress here that this is a very natural way to endow $V^*$ with bilinear pairings, in a sense made precise by Remark \ref{Natural} below.\\
\begin{remark}\label{Natural}
The functors $\mathcal{CS}^t$ and $\mathcal{CS}$ give a (dual) adjoint equivalence, with the unit and counit induced by the usual evaluation map $V \rightarrow V^{**}$. In particular, $\mathcal{CS}$ becomes self-dual adjoint when we restrict to  symmetric bilinear forms.
\end{remark}

\begin{remark}\label{ScalarExt}
Notice that $\mathcal{CS}$ commutes with scalar extension, i.e., if $k'/k$ is a field extension, then the following diagram  commutes:
\begin{center}
  \begin{tikzpicture}
\matrix(m)[matrix of math nodes,
row sep=2.6em, column sep=4.5em,
text height=1.5ex, text depth=0.25ex]
{{\rm FVectB}_k&{\rm FVectB}_k\\
{\rm FVectB}_{k'}&{\rm FVectB}_{k'}\\};
\path[->,font=\scriptsize,>=angle 90]
(m-1-1) edge node[auto] {$\mathcal{CS}$} (m-1-2)
edge node[auto] {$(-)\otimes_k k'$} (m-2-1)
(m-1-2) edge node[auto] {$(-)\otimes_k k'$} (m-2-2)
(m-2-1) edge node[auto] {$\mathcal{CS}$} (m-2-2);
\end{tikzpicture}.  
\end{center}

\end{remark}

 Let $\{v_1,\ldots,v_n\}$ be an ordered $k$-basis of $V$ and let $\{f_1,...,f_n\}$ be its dual basis. We say that $\{v_1^*, \ldots, v_n^*\} \subseteq V$ is {\it the dual  basis of $\{v_1,\ldots,v_n\}$ with respect to $B$} if $f_i=B(v_i^*,\cdot)$ for each $i$. The next lemma gives us a concrete useful formula to compute Casimir parings.

\begin{lemma}\label{ExplicitFormula}
Let $(V,B) \in {\rm FVectB}_k$ and let  $\{v_1,\ldots,v_n\}$ be a $k$-basis of $V$. Then, 
\[ \langle \psi, \phi \rangle_{B} =\sum_{i=1}^{n} \psi(v_i^{*}) \phi(v_i),\] 
where $\{v_1^*, \ldots, v_n^*\}$ is the dual  basis of $\{v_1,\ldots,v_n\}$ with respect to $B$. In particular, the above representation of $ \langle \psi, \phi \rangle_B $ is independent of the choice of basis $\{v_1,\ldots,v_n\}$.
\end{lemma}
\begin{proof}
Let  $\{f_1,\ldots,f_n\}$ be the dual basis $\{v_1,\ldots,v_n\}$. Since both sides of the equality are bilinear on $\phi$ and $\psi$, it is enough to check it when $\phi=f_r$ and $\psi=f_s$, for each $1\leq r,s \leq n$. In this case, $\sum_{i=1}^{n} \phi(v_i^{*}) \psi(v_i)= f_r(v_s^*)\cdot 1$ and $\langle \psi, \phi \rangle_B=B(v_r^*,v_s^*)=f_r(v_s^*)$.
\end{proof}

Now that we have a natural way to get  pairings on $V^*={\rm Hom}_k(V,k)$, we can use it to get parings on more general hom sets ${\rm Hom}_k(V,W)$. We do this  by identifying ${\rm Hom}_k(V,W)$ with $V^* \otimes_k W$ via the usual isomorphism; taking  $f\otimes w \in V^* \otimes_k W$ to the map $v \mapsto f(v)w$.

\begin{definition}
Let $k$ be a field and $(V,B) \in {\rm FVectB}_k $.
\begin{enumerate}[(i)]
    \item Given $k$-bilinear map of vector spaces $V_1 \times V_2 \rightarrow V_3$, the associated Casimir pairing is the map \[{\rm Hom}_k(V,V_1) \otimes_k {\rm Hom}_k(V,V_2)  \rightarrow V_3,\] obtained by taking the tensor product \[ \langle \cdot , \cdot\rangle_{B} \otimes m: (V^* \otimes_k V_1) \otimes_k (V^* \otimes_k V_2) \to (k \otimes_k V_3) \]

    of the maps  $\langle \cdot , \cdot\rangle_{B}:V^* \otimes_k V^* \rightarrow k$ and $m: V_1 \otimes_{k} V_{2} \to V_{3}$.
    \item When $V_1=V_2=V_3=R$ is an $k$-algebra and $R\times R \rightarrow R$ is the multiplication map, we call the corresponding pairing
    \[ {\rm Hom}_k(V,R) \otimes_k {\rm Hom}_k(V,R)  \rightarrow R,\]
    the Casimir paring on ${\rm Hom}_k(V,R)$ associated to $B$.
\end{enumerate}
\end{definition}
\begin{remark}
More explicitly, Lemma \ref{ExplicitFormula} shows that the Casimir paring associated to $V_1 \otimes V_2 \rightarrow V_3$ is the map that takes $ \psi\otimes \phi \in  {\rm Hom}_k(V,V_1) \otimes_k{\rm Hom}_k(V,V_2)$ to the image of the element
\[\sum_{i=1}^{n} \psi(v_i^{*}) \otimes \phi(v_i) \in V_1\otimes V_2\]
under  $V_1 \otimes_k V_2 \rightarrow V_3$. Where as before $\{v_i\}$ is any $k$-basis of $V$ and $\{v_i^*\}$ its dual with respect to $B$. 
\end{remark}

By slight abuse of notation for any $k$-algebra $R$ we  still denote the associated Casimir paring on ${\rm Hom}_k(V,R)$ by $\langle \cdot , \cdot\rangle_{B}$.

\begin{example}\label{ElCasimiro}
Let $\mathfrak{g}$  be an $n$-dimensional Lie algebra over a field $k$ with $\text{char}(k)=0$. By Cartan's criterion the Killing form $B$ on $\mathfrak{g}$ is nondegenerate if and only if $\mathfrak{g}$ is semisimple. In that case, if we take $V=\mathfrak{g}$, $R=U(\mathfrak{g})$ its universal enveloping algebra and $\iota :\mathfrak{g} \hookrightarrow U(\mathfrak{g})$ the canonical inclusion, then \[C:=\langle \iota, \iota \rangle_B \in  U(\mathfrak{g})\]

\noindent is the well known quadratic Casimir element of $\mathfrak{g}$. 
Despite its simplicity, the Casimir element turns out to play a central role in the modern proofs  of several key theorems in the representation theory of Lie algebras such as  Weyl's theorem on complete reducibility (see \cite[Section 6.3]{Humphreys}), as well as in  Weyl's and Konstant's character formulas as shown in the  appendix to chapter VI in \cite{Humphreys}. 
\end{example}

\begin{example} Let $(M,g)$ be a Riemannian manifold. Since the metric $g$ is everywhere nondegenerate, it gives isomorphisms between the tangent bundle $TM$ and its dual $T^*M=\Omega^1(M)$ known as the musical isomorphisms $\flat$ and $\sharp$. One precisely uses  $\sharp: \Omega^1(M) \rightarrow TM$ to induce an inner product on $\Omega^1(M)$, i.e., the usual inner product of $1$-forms is pointwise the Casimir pairing on $\Omega^1_p(M)={\rm Hom}(T_pM, \mathbb{R})$ associated to $g_p$. 

\end{example}

\begin{proposition}\label{FunctorialityCasimir} For a fixed $(V,B) \in {\rm FVectB}_k$, the formation of Casimir pairings is functorial on $V_1\otimes_k V_2 \rightarrow V_3$. More precisely, for each commutative diagram

\begin{center}
  \begin{tikzpicture}
\matrix(m)[matrix of math nodes,
row sep=2.6em, column sep=4.5em,
text height=1.5ex, text depth=0.25ex]
{V_1\otimes_k V_2&V_3\\
V_1'\otimes_k V_2'&V_3'\\};
\path[->,font=\scriptsize,>=angle 90]
(m-1-1) edge node[auto] {} (m-1-2)
edge node[left] {$f\otimes g$} (m-2-1)
(m-1-2) edge node[auto] {} (m-2-2)
(m-2-1) edge node[auto] {} (m-2-2);
\end{tikzpicture},  
\end{center}

\noindent the diagram  on Casimir pairings
\begin{center}
  \begin{tikzpicture}
\matrix(m)[matrix of math nodes,
row sep=2.6em, column sep=4.5em,
text height=1.5ex, text depth=0.25ex]
{{\rm Hom}_k(V,V_1) \otimes_k{\rm Hom}_k(V,V_2)&V_3\\
{\rm Hom}_k(V,V_1') \otimes_k{\rm Hom}_k(V,V_2')&V_3'\\};
\path[->,font=\scriptsize,>=angle 90]
(m-1-1) edge node[auto] {} (m-1-2)
edge node[left] {$f_*\otimes g_*$} (m-2-1)
(m-1-2) edge node[auto] {} (m-2-2)
(m-2-1) edge node[auto] {} (m-2-2);
\end{tikzpicture}  
\end{center}

\noindent is also commutative.
\end{proposition}

\begin{proof} This is already encoded in the definition since the identification we are using $V^*\otimes_k W ={\rm Hom}_k(V,W)$ is functorial on $W$.\end{proof}

The following statements contain the main features of the  Casimir pairings that will be needed latter on.

\begin{corollary}\label{SomeProp} Let $(V,B) \in {\rm FVectB}_k$, $R$ be a $k$-algebra and $\phi,\psi \in {\rm Hom}_k(V,R)$.
\begin{enumerate}[(i)]
\item Functoriality on $R$: If $\theta:R \rightarrow R'$ is a homomorphism of $k$-algebras then \[\theta\left( \langle \psi,\phi \rangle_B\right)=\langle \theta \circ\psi, \theta \circ \phi \rangle_B.\]
\item Scalar extension: If $k'/ k$ is a field extension, then $$\langle \psi \otimes 1,\phi\otimes  1 \rangle_{B\otimes k'}=\langle \psi,\phi \rangle_B\otimes 1 \in R\otimes_k k'.$$
\end{enumerate}
\end{corollary}

\begin{proof}
A homomorphism of $k$-algebras is a map $\theta$ making the diagram 
    \begin{center}
  \begin{tikzpicture}
\matrix(m)[matrix of math nodes,
row sep=2.6em, column sep=4.5em,
text height=1.5ex, text depth=0.25ex]
{R\otimes_k R& R\\
R'\otimes_k R'&R'\\};
\path[->,font=\scriptsize,>=angle 90]
(m-1-1) edge node[auto] {} (m-1-2)
edge node[left] {$\theta\otimes \theta$} (m-2-1)
(m-1-2) edge node[auto] {$\theta$} (m-2-2)
(m-2-1) edge node[auto] {} (m-2-2);
\end{tikzpicture}  
\end{center}
\noindent commutative. Thus part $(i)$ this follows from the functoriality claimed in Proposition \ref{FunctorialityCasimir}. Part $(ii)$ is a consequence of the compatibility in Remark \ref{ScalarExt}. 
\end{proof}

\begin{corollary}\label{FaithfullFunc} Let $R$ be an algebra over a field $k$.

\begin{enumerate}[(i)]
    \item  The functor from ${\rm FVectB}_k$ to the category of $R$-modules with bilinear forms taking $(V,B)$ to  $R$-module ${\rm Hom}_k(V,R)$ endowed with its Casimir pairing $\langle \cdot,\cdot \rangle_B$ and a map $\phi:(V,B) \rightarrow (V',B')$ to $\phi^*$ is a contravariant embedding. 
    \item  Let $(V,B)$ and $(V',B')$ be objects in  ${\rm FVectB}_k$.  A $k$-linear map $\phi:V \rightarrow V'$ is a bijective isometry of $k$-spaces if only if \[\phi^*:{\rm Hom}_k(V',R) \rightarrow {\rm Hom}_k(V,R) \] is a bijective isometry of $R$-modules, where we endow the hom sets with the associated Casimir pairings. 
\end{enumerate}
  
\end{corollary}
\begin{proof}

Since $k$ is a field, every $k$-algebra is faithfully flat over $k$. Hence the scalar extension $(-)\otimes_k R$ from ${\rm FVectB}_k$ to the category of $R$-modules with bilinear forms
is an embedding. The functor in $(i)$ is the composition of this with the duality $\mathcal{CS}$. Since we can tell whether or not an arbitrary $k$-linear map is a bijective isometry after a faithfully flat extension of scalars, this  factorization also proves  $(ii)$.
\end{proof}

More importantly for our purposes, when ${\rm char}(k) \neq 2$ and $R$ is a commutative $k$-algebra, restricting the functor in part $(i)$ to symmetric bilinear forms yields a contravariant embedding 
\[ {\rm QSp}_k^* \hookrightarrow {\rm QMod}_R,\]
from the category nondegenerate quadratic $k$-spaces and (bijective) isometries to the category of quadratic $R$-modules.

\subsection{Trace forms}
\begin{definition}
Let $k$ be a field and let $E$ be an \'etale algebra over $k$.  For any $k$-algebra $R$ we denote by  $\langle \cdot, \cdot\rangle_{{\rm tr}_{E/k}}$ the Casimir pairing on ${\rm Hom}_{k}(E,R)$ associated to the non-degenerate symmetric bilinear form given by the trace pairing 
\begin{displaymath}
\begin{array}{cccc}
 \mathrm{tr}_{E/k}: & E \times E &\rightarrow& k  \\  & (x,y) & \mapsto &
\mathrm{Tr}_{E/k}(xy).
\end{array}
\end{displaymath}
\end{definition}

\begin{lemma}\label{TracePairingOrtonormalBasis}
Let $E$ be an \'etale algebra over $k$. Suppose that $\Omega/k$ is a separable field extension that splits $E$, i.e., such that $E \otimes_{k} \Omega \cong \Omega^n$.  Then  $\{\sigma_1,\ldots,\sigma_n\}:={\rm
Hom }_{k-alg}(E,\Omega)$ is an orthonormal $\Omega$-basis of ${\rm Hom}_{k}(E,\Omega)$.
\end{lemma}

\begin{proof} In the case $E=k^n$ the algebra is already split over $k$, the trace form  is the usual dot product on $k^{n}$ and the corresponding isomorphism $k^n \rightarrow {\rm Hom}_k(k^n,k)$ takes the standard $i$-th basis vector $e_i$ to the $i$-th projection $\pi_i:k^n \rightarrow k$. Since $\{e_{1},\cdots, e_{n}\}$ is an orthonormal basis of $k^n$, the set $\{\pi_{1},\cdots, \pi_{n}\}$ is an orthonormal basis of $ {\rm Hom}_k(k^n,k)$. In the general case, applying ${\rm Hom}_\Omega(-,\Omega)$ to an isomorphism $\sigma: E \otimes_k\Omega \rightarrow \Omega^n$,  we get an  isometry of $\Omega$-spaces

\[\Omega^n \rightarrow {\rm Hom}_\Omega(\Omega^n,\Omega) \xrightarrow{\sigma^*} {\rm Hom}_\Omega(E \otimes_k \Omega,\Omega)= {\rm Hom}_{k}(E,\Omega), \]
\noindent which takes the canonical basis $\{e_1,\ldots,e_n\}$ of $\Omega^n$ to $\{\sigma^*(\pi_1),\ldots, \sigma^*(\pi_n)\}=\{\sigma_1, \ldots, \sigma_n\}$.
  \end{proof}

It follows from the above that any $k$-linear map from $E$ to $\Omega$ can be written in terms of Casimir elements: 
\begin{corollary}\label{FourierCoefficients}
Let $k,E, \Omega$ and $\{\sigma_1,\ldots,\sigma_n\}$ be as above. Then, for all $\phi \in {\rm Hom}_{k}(E,\Omega)$ \[\phi=\sum_{i=1}^{n} \langle \phi , \sigma_i\rangle_{{\rm tr}_{E/k}} \,\sigma_i.\]
\end{corollary}

\begin{remark}\label{GaloisEtale}
Grothendieck's formulation of Galois theory gives a duality between \'etale $k$-algebras and finite $\mathfrak{G}_k$-sets, where $\mathfrak{G}_k={\rm Gal}(k_s/k)$ and $k_s$ is a separable closure of $k$. There, an \'etale algebra $E$ corresponds precisely to the $\mathfrak{G}_k$-set ${\rm
Hom }_{k-alg}(E,k_s)$. The fact that ${\rm
Hom }_{k-alg}(E,k_s)$  also happens to be an orthonormal $k_s$-basis for ${\rm
Hom }_{k}(E,k_s)$ allows us to leverage this correspondence. For instance, a decomposition of $E$ as a product $\prod_i L_i$ of separable extensions $L_i$ amounts to  a decomposition into transitive $\mathfrak{G}_k$-sets  \[{\rm
Hom }_{k-alg}(E,k_s)= \bigsqcup_i {\rm
Hom }_{k-alg}(L_i,k_s),\] which in turn induces an orthonormal decomposition of $k_s$-spaces
\[{\rm
Hom }_{k}(E,k_s)= \bigot_i {\rm
Hom }_{k}(L_i,k_s).\]

\end{remark}

\begin{corollary}\label{TheMatrixU}
Let $k$ be a field and let $E/k$ and $E'/k$ be  \'etale algebras same dimension $n$. Suppose that $\Omega/k$ is a separable field extension splitting both $E$ and $E'$. Let $\phi: E \rightarrow E'$ be an $k$-linear map. Then the following are equivalent:
\begin{enumerate}[(i)]
\item The map $\phi$ is an isometry between $(E,{\rm tr}_{E/k})$ and $(E',{\rm tr}_{E'/k})$.
\item The map $\phi^*$ is an isometry between the $\Omega$-spaces ${\rm Hom}_k(E,\Omega)$  and ${\rm Hom}_k(E',\Omega)$.
\item The matrix $U=\left(c_{ij} \right)$ is orthogonal, where $c_{ij}:=\langle \sigma_i, \tau_j\phi \rangle_{{\rm tr}_{E/k}}$, and  $\{\sigma_i\}$, $\{\tau_i\}$ are the sets of homorphisms of $k$-algebras of $E$ and $E'$ into $\Omega$.
\end{enumerate}
\end{corollary}

\begin{proof}
The equivalence of $(i)$ and $(ii)$ is a particular case of Corollary \ref{FaithfullFunc}$(ii)$,  and the equivalence of $(ii)$ and $(iii)$ follows from the fact that $\{\sigma_i\}$ and $\{\tau_i\}$ are orthonormal bases by Lemma \ref{TracePairingOrtonormalBasis}.  
\end{proof}

\section{Linear disjointness}\label{Disjoint}

When studying separable extensions that are trace isometric we discovered that linear disjointness makes matters simpler; more specifically in this case the entries of the matrix of Corollary \ref{TheMatrixU}(iii) are all in the conjugacy orbit of a single Casimir value. Also, when dealing with $S_{n}$ number fields linear disjointness is equivalent to not being conjugate. In this section we explain all this in detail.

\begin{corollary}\label{LinDisjOrthoMatr}
Let $K/k$ and $L/k$ be separable field extension, and suppose $\Omega/k$ is a separable field extension containing a Galois clousure of $KL/k$. Let us denote by $\iota_{K}$ and $\iota_{L}$ the inclusions from $K$ (resp $L$) into $KL$. Suppose that \[\phi: (K,{\rm tr}_{K/k}) \xrightarrow{} (L,{\rm tr}_{L/k})\] is an isometry and let $c \in KL$ be the Casimir element $\displaystyle c:=\langle \iota_K ,\iota_L \phi \rangle_{{\rm tr}_{K/k}}.$ If $K$ and $L$ are linearly disjoint over $k$, then there is an indexing \[{\rm Hom}_{k-alg}(KL, \Omega)= \{\theta_{i,j};   1 \leq i, j\leq n\}\] of the set of $k$- embeddings of $KL$ into $\Omega$ such that the matrix $U:=\theta_{i,j}(c)$ is orthogonal.
\end{corollary}

\begin{proof}
Since $K$ and $L$ are $k$-linearly disjoint the restriction map \[{\rm Hom}_{k-alg}(KL, \Omega)  \to  {\rm Hom}_{k-alg}(K, \Omega) \times {\rm Hom}_{k-alg}(L, \Omega); \hspace{.1in}  \theta  \mapsto  (\theta \circ \iota_{K}, \theta \circ \iota_{L}) 
\] is a bijection. The result follows from this together with Proposition \ref{SomeProp}(i) and ${\rm Corollary}$ \ref{TheMatrixU} (iii). \end{proof}


\subsection{Linear disjointness and an isomorphism criterion for number fields}

When looking for arithmetic invariants of number fields it is always important to have in hand a criterion to decide whether or not two fields are conjugate. As it turns out, for $S_n$ number fields one such a criterion is that the fields are linearly disjoint. Here we prove something a little bit more general.

\begin{proposition}\label{SnDisjoint}
Let $K,L$ be number fields of the same degree $n$ such that $\textrm{Gal}(\widetilde{K}/\mathbb{Q}) \cong S_n$ ($\widetilde{K}$ the Galois closure of $K$) and $L/\mathbb{Q}$ has no intermediate extensions. Then, $K \not \cong L$ if and only if $K/\mathbb{Q} \textnormal{ and } L/\mathbb{Q}$  are linearly disjoint.
\end{proposition}
\begin{proof}
{\it ``if'' :} Suppose $K$ and $L$ are linearly disjoint, then $K \not \cong L$. Otherwise,
\[K\cong L \cong \mathbb{Q}[X]/(p(X))\]
for some $p(X) \in \mathbb{Q}[X]$ and tensoring with $L$ we get $K\otimes_{\mathbb{Q}} L \cong L[X]/(p(X)) $
which is not a field, as $p(X)$ has a root in $L$.\\

{\it ``Only if'' :} Suppose $ K \not \cong L$. To prove that $K$ and $L$ are linearly disjoint it is enough to show that so are $\widetilde K$ and $L$. Since $L/\mathbb{Q}$ has no intermediate extensions and $\widetilde{K}/\mathbb{Q}$ is Galois, then $L/\mathbb{Q}$ and $\widetilde{K}/\mathbb{Q}$ are linearly disjoint if and only if $L\cap \widetilde{K}= \mathbb{Q}$, i.e., if and only if $L \not \subset \widetilde{K}$.\\

\noindent Now, if $n \neq 6$, then every subgroup of $\textrm{S}_n$ of index $n$ is conjugate to $H:=\{ \sigma \in \textrm{S}_n : \sigma(1)=1\}$ (see \cite[Lemma 7.8.5]{RSchSubgr}). Therefore, in that case  $L \not \subset \widetilde{K}$, otherwise, the subgroups  $H_L:=\textrm{Gal}(\widetilde{K}/L)$ and $H_K:=\textrm{Gal}(\widetilde{K}/K)$ of index $n$ in $\textrm{Gal}(\widetilde{K}/\mathbb{Q})\cong S_n$  would be conjugate, contrary to our assumption.\\

\noindent If $n=6$ it turns out, see \cite[Lemma 7.8.6]{RSchSubgr}, that for the only two conjugacy classes $\Delta_1$ and $\Delta_2$ of index $6$ in $\textrm{S}_6$, each pair of representatives $H_1\in \Delta_1$ and $H_2 \in \Delta_2$ satisfy \[36=[\textrm{S}_6:H_1\cap H_2]=[\textrm{S}_6:H_1][\textrm{S}_6:H_2]=6 \cdot 6.\] If $K$ and $L$ are not linearly disjoint then, arguing as above, $L  \subset \widetilde{K}$ and since $K$ and $L$ are  not isomorphic the groups $H_L$ and $H_K$ belong to the two different conjugacy classes of index $6$ in $\textrm{Gal}(\widetilde{K}/\mathbb{Q})\cong S_6$. By the group theoretic fact above we conclude that $[KL:\Q]=[L:\Q][K:\Q]$ which contradicts that assumption that $K$ and $L$ are not linearly disjoint.
\end{proof}

\begin{corollary}\label{FundDiscLinDisj}
   Let $K,L$ be degree $n$, $S_{n}$ number fields. Then,
   $K \not \cong L$ if and only if $K/\mathbb{Q} \textnormal{ and } L/\mathbb{Q} \textnormal{ are linearly disjoint}$
\end{corollary}  
   \begin{proof}
   Since one point stabilizers are maximal subgroups of $S_{n}$, see \cite[I, Corollary 1.5A]{dixon}, it follows from Galois correspondence that $K/\Q$ and $L/\Q$ have no intermediate extensions. The result follows from Proposition \ref{SnDisjoint}. \end{proof}

\section{Casimir elements associated to the trace pairing over number fields}\label{CasimiroYLosNumberFields}

In this section we focus our attention on questions about the integrality of Casimir values associated to the trace pairing. The strategy here is a game of local vs global, where we use in our favor that we know well the local behavior of the integral trace for tame extensions (see \cite[\S3]{Man9}.)

\subsection{Local considerations}

Here we use the standard decomposition of the trace as local traces to obtain information on Casimir pairings from their local counterparts.

\subsubsection{$p$-integrality of Casimir Pairings.}

\begin{lemma}\label{LocalTraces}
Let $K/\Q_p$ and $L/\Q_p$ be  finite field extensions with ramification indexes $e:=e(K)$ and $e(L)=:e'$, and different ideals $\mathcal{D}_K$ (resp. $\mathcal{D}_L$). Suppose $\sigma$ and $\tau$ are $\Q_p$-embedings  of $K$ and $L$ into $\overline{\Q}_p$, and $\psi: K \rightarrow L$ is a $\Q_p$-linear map taking $O_K$ into $O_L$ and $\mathcal{D}_K^{-1}$ into $\mathcal{D}_L^{-1}$.
If either both extensions are tamely ramified or $e=e'=2$, then the element \[e p^{1-\left(\frac{1}{e}+\frac{1}{e'}\right)} \langle \psi^*(\tau), \sigma \rangle_{{\rm tr}_{K/\Q_p}}\] is $p$-integral. If $K$ is unramified, then $\langle \psi^*(\tau), \sigma \rangle_{{\rm tr}_{K/\Q_p}}$ is $p$-integral.
\end{lemma}

\begin{proof}
Let $F$ be the maximal unramified subextension of $K$. First let us assume that $L/\Q_p$ and $K/\Q_p$ are tamely ramified. Then by \cite[Lemma 3.1]{Man9}, there exists a uniformizer $\pi$ of $O_K$ and a unit $\mu_F\in O_F^*$ such that $\pi^e=\mu_Fp$ and $O_K=O_F[\pi]$. This implies that ${\rm Tr}_{K/F}(\pi^m)=0$ whenever $e\nmid m$, hence the dual $F$-basis with respect to the trace pairing ${\rm tr}_{K/F}$ of the integral basis $ \{1,\ldots,\pi^i,\ldots \pi^{e-1}\}$ for $O_K$ over $O_F$ is $\{\frac{1}{e},\ldots,\frac{1}{e}\pi^{-i}, \ldots, \frac{1}{e}\pi^{-(e-1)}\}$. \\

Let $\{w_1,\ldots,w_f\}$ be any integral basis for $O_F$ over $\Z_p$ with dual basis $\{w_1^*,\ldots,w_f^*\}$ with respect to the trace pairing ${\rm tr}_{F/\Q_p}$. The set $\{w_i\pi^j:1\leq i\leq f, 0\leq j<e \}$ is an integral basis for $O_K$ over $\Z_p$ with dual basis $\{\frac{1}{e}w_i^*\pi^{-j}:1\leq i\leq f, 0\leq j<e \}$ with respect to trace pairing ${\rm tr}_{K/\Q_p}$. Computing the pairing in basis  we get the sum
\begin{equation*}
 \langle \psi^*(\tau), \sigma \rangle_{{\rm tr}_{K/\Q_p}}=\sum_{i,j} \tau\psi(\frac{w_i^*}{e\pi^j})\sigma(w_i\pi^j).   \tag{*}
\end{equation*}
 When $j=0$, the summand $\tau\psi(\frac{w_i^*}{e}) \sigma(w_i)$ is already $p$-integral. This is because $\psi$ maps $O_K$ into $O_L$ and $F/\Q_p$ is unramified, so each $w_i^* \in \mathcal{D}_F^{-1}=O_F$. While if $j\geq 1$, since by hypothesis $x:=\psi(\frac{w_i^*}{e\pi^j}) \in \mathcal{D}_L^{-1}$ and $L$ is tame, $v_p(x) \geq -\frac{e'-1}{e'}$. Thus
\[v_p( x\sigma(w_i\pi^j) )\geq-\frac{e'}{e'-1}+\frac{j}{e}\geq  -\frac{e'-1}{e'}+\frac{1}{e}=-1+\left(\frac{1}{e}+\frac{1}{e'}\right).\]
Hence $v_p(\langle \psi^*(\tau) ,\sigma\rangle ) \geq -1+\left(\frac{1}{e}+\frac{1}{e'}\right)$, as claimed. \\

If now $e=e'=2$, and of the extensions is wild, then $p=2$ and $O_K$ still has basis  of the form $\{1,\pi\}$ over $O_F$ such that $\pi^2 \in O_F$; but now $\pi$ is either a uniformizer or a unit. In the former case the same argument above applies, with the adjustment that we need to multiply by $e=2$ and replace the tameness of $L$ by the inequality $v_2(\mathcal{D}_L)\leq \frac{e'-1}{e'}+v_2(e')=\frac{3}{2}$; in the latter case $(*)$ directly shows that $2\langle \psi^*\tau, \sigma \rangle$ is $2$-integral. \\

Finally, if $K=F$ is unramified, then $(*)$ holds with $e=1$ and shows that $\langle \psi^*(\tau), \sigma \rangle$ is $p$-integral.
\end{proof}

\begin{proposition}\label{pInt}
Let $K,L$ be  number fields of degree $n$, and let  $\phi: (K,{\rm tr}_{K/\mathbb{Q}}) \rightarrow (L,{\rm tr}_{L/\mathbb{Q}})$ be an isometry such that $\phi(O_K) = O_L$. Consider the matrix $U$ representing the $\Q$-linear map
\[ \phi^*: {\rm Hom}_{\Q}(L,\overline{\Q}) \rightarrow  {\rm Hom}_{\Q}(K,\overline{\Q})\]
in the bases that consists of the $\Q$-embeddings of $K$ and $L$, and let $p$ be a rational prime ramified in $K$.  
\begin{enumerate}[(i)]
    \item  If $L/\Q$ and $K/\Q$ are both tamely ramified at $p$, then $p^{1-\frac{2}{e}}U$ has $p$-integral entries. Where $e$ is the largest ramification index of a prime in $L$ or $K$ dividing $p$.
    \item If $p=2$, and $L/\Q$ and $K/\Q$ are both tamely ramified at $2$ except for possibly some wild primes with ramification index $2$, then $2U$ has $2$-integral entries.
\end{enumerate}
\end{proposition}

\begin{proof}
Since we are only concerned with $p$-integrality, scalar extension compatibility (Proposition \ref{SomeProp}(ii)) allows us to replace $U$ with the matrix $U_p$ representing the $\overline{\Q}_p$-linear map
\[ (\phi\otimes 1)^*: {\rm Hom}_{\Q_p}(L\otimes \Q_p,\overline{\Q}_p) \rightarrow  {\rm Hom}_{\Q_p}(K\otimes \Q_p,\overline{\Q}_p)\]
in the bases that consists of $\Q_p$-algebra morphisms of $K\otimes \Q_p$ and $L \otimes \Q_p$  into $\overline{\Q}_p$. As noted in Remark \ref{GaloisEtale}, the decomposition of $\Q_p$-\'etale algebras $K\otimes \Q_p = \prod_{v \mid p} K_v$ and  $L\otimes \Q_p = \prod_{w \mid p} L_w$ induces a orthogonal decomposition of $\overline{\Q}_p$-spaces


\[{\rm
Hom }_{\Q_p}(K\otimes \Q_p ,\overline{\Q}_p)= \bigot_{v\mid p} {\rm
Hom }_{\Q_p}(K_v,\overline{\Q}_p)\]
\[{\rm
Hom }_{\Q_p}(L\otimes \Q_p,\overline{\Q}_p)= \bigot_{w\mid p} {\rm
Hom }_{\Q_p}(L_w,\overline{\Q}_p).\]

This shows that it is enough to consider the pairings $\langle (\phi \otimes 1)^*(\tau)|_{K_v},\sigma \rangle_{{\rm tr}_{K_v/\Q_p}}$ for any $\sigma \in  {\rm
Hom }_{\Q_p}(K_v,\overline{\Q}_p)$ and $\tau \in  {\rm
Hom }_{\Q_p}(L_w,\overline{\Q}_p)$: Since for $v' \neq v$, $\sigma|_{K_{v'}}=0$ the corresponding entry of $U_p$ is 
\[ \langle (\phi \otimes 1)^*(\tau),\sigma \rangle_{{\rm tr}_{K \otimes \Q_p/\Q_p}}=\sum_ {v' \mid p} \langle (\phi \otimes 1)^*(\tau)|_{K_{v'}},\sigma|_{K_{v'}} \rangle_{{\rm tr}_{K_{v'}/\Q_p}}= \langle (\phi \otimes 1)^*(\tau)|_{K_v},\sigma \rangle_{{\rm tr}_{K_v/\Q_p}}.\] 
Now since $\phi\otimes 1$ is an integral isometry it must map the dual of $O_K \otimes \Z_p= \prod_{v|p} O_{L_v}$ in $K\otimes \Q_p$ into the dual of $O_L\otimes \Z_p= \prod_{w|p} O_{L_v}$ in $L \otimes \Q_p$, i.e.,
 \[\phi\otimes 1:\prod_{v|p} \mathcal{D}_{K_v}^{-1} \rightarrow  \prod_{w|p} \mathcal{D}_{L_w}^{-1}\]
It follows that the local extensions $K_v/\Q_p$, $L_w/\Q_p$ and the map \[\psi:K_v \xrightarrow[]{(\phi \otimes 1)|_{K_v}} \prod_{w' \mid p} L_{w'} \rightarrow L_w\] satisfy the hypothesis of Lemma \ref{LocalTraces}. Since $(\phi \otimes 1)^*$ is an isometry and the same considerations apply to $\phi^{-1}$, the symmetry $\langle (\phi\otimes 1)^* (\tau), \sigma  \rangle=\langle  \tau, (\phi^{-1}\otimes 1)^*(\sigma)  \rangle$ allows to switch $K_v$ and $L_w$ if necessary and thus reduce to the case when $e(K_v) \leq e(L_w)$. Applying the estimates in Lemma \ref{LocalTraces} to  \[\langle (\phi \otimes 1)^*\tau|_{K_v},\sigma|_{K_v} \rangle_{{\rm tr}_{K_v/\Q_p}}=\langle \psi^*(\tau), \sigma \rangle_{{\rm tr}_{K_v/\Q_p}}\]  finishes the proof in each case. 

\end{proof}

\subsection{Integrality of Casimir elements}

Here we give some conditions to show how far the casimir pairing associated to an isometry of the trace pairings of a couple of number fields is integral.

\begin{definition}
Let $K,L$ be number fields with fixed embeddings $\iota_K$ and $\iota_L$ into $\C$. Given $\phi:K\rightarrow L$ a $\mathbb{Q}$-linear map  we denote by $c_{\phi}$ the Casimir element given by \[c_{\phi}:=\langle \iota_K, \iota_L \phi \rangle_{{\rm tr}_{K/\Q}}\] and we define $M_{\phi}$ to be the least positive integer $m$ such that $ m c_{\phi}$  is an algebraic integer.
\end{definition}

\noindent The {\it radical} of a non zero integer $m$, i.e., the product of all prime divisors of $m$, is denoted by ${\rm rad}(m)$.

\begin{theorem}\label{RadDiscriminantCasimirAlgInteger}
Let $K,L$ be  number fields of degree $n$, and let  $\phi: (K,{\rm tr}_{K/\mathbb{Q}}) \rightarrow (L,{\rm tr}_{L/\mathbb{Q}})$ be an isometry such that $\phi(O_K) = O_L$. Then,
\begin{enumerate}[(i)]
    \item The integer $M_{\phi}$ divides $d_s(K,L)$.
    \item If $K$ is tamely ramified, then $M_{\phi} \mid {\rm rad} (d_s(K,L))$.
    \item If $K$ and $L$ are only wildly ramified at $2$ and $e_2(K)=2=e_2(L)$, then $M_{\phi} \mid {\rm rad} (d_s(K,L))$.
    
\end{enumerate}

\end{theorem}

\begin{proof}

Let $(\alpha_1,\ldots,\alpha_n)$ be an integral basis of $K$. Recall that \[c_{\phi}=\sum_{i=1}^{n} \alpha_i^{*} \phi(\alpha_i).\]
Since each $\alpha_i^{*}$ belongs to the codifferent ideal $\mathcal{D}_K^{-1}$ and $[\mathcal{D}_K^{-1}:O_K]=d_K$, we have that $d_Kc_{\phi}$ is integral. Let $p$ be a rational prime. If $p$ is not ramified in $K$ then $c_{\phi}$ is $p$ integral since $d_Kc_{\phi}$ is, in particular $d_{s}c_{\phi}$ is $p$ integral. Suppose that $p$ is ramified. 

\begin{itemize}
    \item If $p\mid d_f(K,L)$, then $e_{p}(K)=e_p(L)=2$. Hence, by Proposition  \ref{pInt}$(i)$  $c_{\phi}$ is $p$-integral in particular $d_{s}c_{\phi}$ is as well. 
    \item If $p \nmid d_f(K,L)$ then $d_{s}c_{\phi}=\frac{d_{K}c_{\phi}}{d_{f}}$ is $p$-integral since $d_{K}c_{\phi}$ is integral. 
\end{itemize}

\noindent To prove the $(ii)$ suppose that the ramification of $p$ in  $K$ is tame. Let $\mathfrak{p}$ be a prime of $K$ lying over $p$. Then \[v_{\mathfrak{p}}(p\mathcal{D}_K^{-1})=e(\mathfrak{p}|p)-(e(\mathfrak{p}|p)-1)=1\geq0.\] It follows that $pc_{\phi}$ is $p$-integral, and we see that ${\rm rad}(d_s(K,L)) c_{\phi}$ is integral. Likewise, if $p=2$ and $e_2(K)=e_2(L)=2$, then Proposition \ref{pInt}$(ii)$ shows that $2 c_{\phi}$ is $2$-integral, proving part $(iii)$.   

\end{proof}

\begin{proposition}\label{CasimirNotZero}
Let $K$ be a totally real number field of degree $n$, and let $L$ be a number field such that the extensions $L/\Q$ and $K/\Q$ are linearly disjoint. Suppose that there is an isometry \[\phi: (O_K,{\rm tr}_{K/\Q}) \rightarrow (O_L,{\rm tr}_{L/\Q}).\] Then, \[\sqrt{n}\leq M_{\phi}.\] Moreover, in such a case, the equality holds if and only if $M_{\phi}c_{\phi}=\pm 1$.
\end{proposition}

\begin{proof}
Since the isometry class of the integral trace determines the {\rm discriminant}, the signature, and the degree of a field we have that $L$ is totally real with the same discriminant and degree as $K$.  Let $m$ be a positive integer such that $\alpha=mc_{\phi}$ is an algebraic integer. Let $U=(c_{ij})$ where $c_{ij}:=\langle \sigma_i, \tau_j\phi \rangle_{{\rm tr}_{K/F}}$, and  $\{\sigma_i\}$, $\{\tau_i\}$ are the sets of complex embeddings of $K$ and $L$. Since $U$ is orthogonal,  thanks to Corollary \ref{TheMatrixU} $(iii)$, it follows that if we let $N=mU$ then $NN^{t}=m^2I$. Taking traces on both sides we obtain
 \[ \sum_{i,j} (mc_{ij})^2= {\rm tr}(NN^{t})=m^2n\]
 Since $K/\Q$ and $L/\Q$ are linearly disjoint, it follows from the construction of Corollary \ref{LinDisjOrthoMatr}, that the $n\times n$ matrix built by the conjugates of $\alpha$ in $KL$, up to permutation of its entries, is equal to $N$. In particular, \[\tr_{KL/\Q}(\alpha^2)=\sum_{i,j} (mc_{ij})^2.\] Since $\alpha$ is integral, and non-zero since $U$ is a non-zero matrix, $N_{KL/\mathbb{Q}}(\alpha^2)\geq 1$. Thanks to the  AM-GM inequality,
 \[n^2=[KL:\Q] \leq \tr_{KL/\Q}(\alpha^2)= m^2 n. \]
 Hence $\sqrt{n}\leq m$ and the equality holds if and only if $\alpha^2=1$.
\end{proof}

\section{Arithmetic consequences}\label{LosResultados}

Two useful and closely related quadratic invariants  have been derived from the integral trace form. They are the trace-zero form and the shape of a number field (an invariant with a more geometric flavor). To define them, let $K$ be a number field. The {\it trace-zero module} of $K$, denoted $O_K^0$ is
\[\{x\in O_K : {\rm tr}_{K/\mathbb{Q}}(x)=0\}.\]
This module and the quadratic form obtained by restricting the trace pairing to it (the trace-zero form) played a role in the study of the trace form for cubic fields in \cite{Man}, and in the work of Ellenberg-Venkatesh on asymptotics of number fields \cite{ev}. We will use it here to cover a crucial case of our main theorem.\\

\noindent The shape of $K$, denoted ${\rm Sh}(K)$, is the class of the lattice obtained by mapping 
\[O_K^\bot=\{x\in \Z+[K:\mathbb{Q}]O_K : {\rm tr}_{K/\mathbb{Q}}(x)=0\}\]
under the Minkowski embedding up to rotation, reflections and multiplication by scalars. When $K$ is totally real, we can equivalently describe ${\rm Sh}(K)$ as the isometry class of the quadratic module obtained by restricting the trace paring to $O_K^\bot$ up to multiplication by $\R^{\times}$. This invariant, introduced for cubic fields in \cite{DTerr}, has been recently studied for several different purposes (see \cite{BH,RH, Man8}).\\

\noindent The following proposition shows that for totally  fields number fields the integral trace form is the strongest of these three invariants.

\begin{lemma}\label{FromTodoTwoZero}
Let $K, L$ be two totally real number fields and suppose $\phi: (O_K,{\rm tr}_{K/\Q}) \rightarrow (O_L,{\rm tr}_{L/\Q})$ is an isometry. Then, $\phi(1)=\pm 1$. In particular, the restriction of $\phi$ induces an isometry between the integral trace zero parts of $K$ and $L$, and also between the shapes of $K$ and $L$.  

\end{lemma}

\begin{proof}
Because the fields have isometric integral trace forms, they must have the same degree, say $n$. Now let $\alpha:=\phi(1)\neq 0$. By hypothesis $\alpha^2$ is a totally positive algebraic integer such that ${\rm tr}_{L/\mathbb{Q}}(\alpha^2)={\rm tr}_{L/\mathbb{Q}}(\phi(1)^2)={\rm tr}_{K/\mathbb{Q}}(1^2)=n$. Hence, from the AM-GM inequality applied to the $n$ conjugates of $\alpha^2 \in L$ over $\mathbb{Q}$, we find that the equality in 
\[ 1 \leq  {\rm N}_{L/\mathbb{Q}}(\alpha^2)^{\frac{1}{n}}\leq \frac{{\rm tr}_{L/\mathbb{Q}}(\alpha^2)}{n}= 1\]
implies $\alpha^2\in \mathbb{Q}$ and thus $\alpha^2=1$. To prove the last claim, note that the isometry $\phi$ maps the elements in $O_K$ orthogonal to $1$ precisely to the elements in $\phi(O_K)=O_L$ orthogonal to $\phi(1)=\pm 1$, i.e., $\phi(O_K^0)=O_L^0$. Similarly, as $\phi(\Z+nO_K)=\phi(1)\Z+nO_L=\Z+nO_L$, we have $\phi(O_K^\bot)=O_L^\bot$. 
\end{proof}

We are now ready to prove the main results of the paper:

\begin{theorem}\label{ElGeneralisimo}
Let $K$ be a degree $n$ totally real $S_{n}$ number field and $L$ be any $S_n$ number field. Let $d_{s}(K,L)$ be the integer defined in \ref{LosEnterosQuePartenDisc}.
\begin{enumerate}[(i)]
    \item Suppose that $n \ge d_{s}(K,L)^{2}$. Then, \[(O_K,{\rm tr}_{K/\mathbb{Q}}) \cong  (O_L,{\rm tr}_{L/\mathbb{Q}}) \ \mbox{if and only if} \ K \cong L.\]
    
    \item If  $K$ does not have wild ramification, then the condition $n \ge d_{s}(K,L)^{2}$ can be improved to $n \ge {\rm rad}(d_{s}(K,L))^{2}$.
    
    \item  If  $K$ and $L$ do not have wild ramification except possibly for some primes with ramification index $2$ lying over $2$, then the condition $n \ge d_{s}(K,L)^{2}$ can also be improved to $n \ge {\rm rad}(d_{s}(K,L))^{2}$ .

\end{enumerate}

\end{theorem}

\begin{proof}
This is clear in degrees $n=1,2$, so let us suppose $n\geq3$. We show the non-trivial implication. Since $K$ and $L$ have isometric integral traces, they share the same degree, signature and discriminant.  If $K$ and $L$ are not isomorphic it follows from Corollary \ref{FundDiscLinDisj} that they are linearly disjoint. Let $\phi$ be an integral isometry, by Corollary \ref{LinDisjOrthoMatr}, there is an orthogonal matrix $U$ whose entries consist of the values $\theta(c_{\phi})$ where $\theta$ runs over the set
${\rm Hom}_{\Q-alg}(E, \C)$. In particular $c_{\phi}\not \in \Q$, otherwise, $\theta(c_{\phi})=c_{\phi}$ for all $\theta$ which is a contradiction since the matrix $U$ is invertible. Since $c_{\phi}$ is not rational it follows from Proposition \ref{CasimirNotZero} and from Theorem \ref{RadDiscriminantCasimirAlgInteger} that 
\[\sqrt{n}< M_{\phi}\leq d_s(K,L).\] The improvement on the upper bound of the above is obtained thanks to  parts $(ii)$ and $(iii)$ of Theorem \ref{RadDiscriminantCasimirAlgInteger}.
\end{proof}

The above together with some previous results on integral traces imply that the integral trace is a complete invariant for real fields with at worst quadratic ramification.

\begin{theorem}\label{TotallyRealFundDiscG}
Let $K,L$ be degree n, $S_n$ number fields. Suppose that $K$ is totally real and that the ramification index of every prime in $K$ and $L$ over $\Q$ is at most $2$. Then, $$(O_K,{\rm tr}_{K/\mathbb{Q}}) \cong  (O_L,{\rm tr}_{L/\mathbb{Q}}) $$ if and only if $K \cong L$.
\end{theorem}
\begin{proof}
This is trivial in degrees $n=1,2$, so let us suppose that $n\geq3$. Thanks to Theorem \ref{ElGeneralisimo}, the result follows for $n \ge 4$. If $n=3$, the restriction on the ramification indexes implies that the Galois closure  $\widetilde{K}$ of $K$ is unramified over $\Q(\sqrt{d_K})$, hence $d_K$ must be fundamental (see \cite{HasseC}). Thus, using Lemma \ref{FromTodoTwoZero} we see that the case $n=3$ follows from \cite[Theorem $6.5$]{Man}.
\end{proof}

\begin{corollary}\label{IntTraceCompleteInv}
The integral trace form is a complete invariant for totally real number fields of fundamental discriminant.
\end{corollary}

\begin{proof}
Since degree $n$ number fields of fundamental discriminant are $S_n$ number fields, see \cite{Tkondo}, and since having fundamental discriminant implies that ramification indices are at most 2, the result follows from Theorem \ref{TotallyRealFundDiscG}.
\end{proof}


\begin{theorem}\label{SqFDiff}
Let $K$ and $L$ be  totally real number fields with square free different ideal. Then any isometry $\phi: (K, {\rm tr}_{K/\mathbb{Q}}) \rightarrow (L, {\rm tr}_{L/\mathbb{Q}})$ such that $\phi(O_K)=O_L$ is equal to plus or minus an isomorphism of fields $K \cong L$.
\end{theorem}

\begin{proof}
From Corollary \ref{TheMatrixU}$(iii)$ we know that $U=(c_{ij})$ is an orthogonal matrix, where $c_{ij}:=\langle \sigma_i , \tau_j \circ \phi \rangle_{{\rm tr}_{K/\mathbb{Q}}}$ and $\{\sigma_1, \ldots, \sigma_n\}$ (resp. $\{\tau_1, \ldots, \tau_n\}$ ) is the set of embeddings of $K$ (resp. $L$). Moreover, because $K$ and $L$ are totally real,  $U\in {\rm M}_n(\mathbb{R})$ and thus $|c_{ij}|\leq 1$ for every $1\leq i,j\leq n$. On the other hand, thanks to Proposition \ref{SomeProp}(ii) we have a natural action of the absolute Galois  group $G_\Q:={\rm Gal}(\overline{\Q}/\mathbb{Q})$  on the entries of $U$. More specifically, if $\theta \in G_\Q$ then  $\theta(c_{ij})=c_{i'j'}$
where $\theta \circ  \sigma_i=\sigma_{i'}$ and $\theta \circ  \tau_j=\tau_{j'}$. In particular, every conjugate of an entry of $U$ is again an entry of $U$ and thus has absolute value  bounded by $1$, i.e.,  $|\theta(c_{ij})|\leq 1$ for every $\theta \in G_{\Q}$ and $1\leq i,j \leq n$.\\

\noindent The fact that the different ideals of $K$ and $L$ are square free implies that $d_s(K,L)=1$, thus each $c_{ij}$ must be an algebraic integer by Theorem \ref{RadDiscriminantCasimirAlgInteger}(i). From the above paragraph we deduce that each $c_{ij}$ is either $0$ or a real root of  unity, or equivalently $c_{ij}\in \{0,\pm 1\}$. Furthermore, $U$ being orthogonal implies that there is only one  $c_{ij}$ equal to $\pm 1$ on each row and on each column of $U$. If $i$ is the index such that $\langle \sigma_i,\tau_1 \circ  \phi \rangle_{\textrm{tr}_{K/\mathbb{Q}}} \neq 0$, then by Corollary \ref{FourierCoefficients} we have
\[\tau_1 \circ  \phi=\pm\sigma_i,\]
and thus $\mp \phi$ is multiplicative.
\end{proof}

\begin{corollary}\label{ElAuto}
Let $K$ be a totally real number field with square free different ideal. Then, \[{\rm Aut} \left((O_K,{\rm tr}_{K/\mathbb{Q}}) \right) \cong 
\Z/2\Z \times {\rm Aut}(K).
\]

In particular, if $n>2$ and $K$ is a $S_{n}$ number field then the automorphism group of the integral trace is ``trivial".

\end{corollary}

\begin{proof}
This follows from Theorem \ref{SqFDiff} by taking $K=L$.
\end{proof}

\begin{corollary}\label{AutTraceFundDisc}

Let $K$ be a real number field of degree at least $3$ and with square free discriminant. Then, \[{\rm Aut} \left((O_K,{\rm tr}_{K/\mathbb{Q}}) \right) \cong 
\Z/2\Z.
\]
\end{corollary}

\begin{proof}
Let $\widetilde{K}$ be the Galois closure of $K$, let $G$ be the Galois group of $\widetilde{K}/\Q$ and let $H$ be the subgroup of $G$ corresponding to $K$. Since $K$ has square free discriminant $G$ is the full symmetric group and $H$ is a one point stabilizer. In particular, $H$ is self-normalizing in $H$ and thus ${\rm Aut}(K)$ is the trivial group. Since square free discriminant implies square free different ideal, the result follows from Theorem \ref{ElAuto}.
\end{proof}

Since random lattices have ``trivial" automorphism group, see for instance \cite{Bier} and \cite{Ban}, this corollary is not completely unexpected. A derived question that comes to mind is whether or not maximal orders of real number fields with trivial automorphism group are random within the space of positive definite lattices.

\begin{remark}

Notice that Corollary \ref{AutTraceFundDisc} could be stated in a more general form: as long as $H$ is self normalizing in $G$, and $K$ has square free different ideal the same conclusion holds. This follows since for arbitrary $K$ we have that  $N_{G}(H)/H \cong {\rm Aut}(K)$.

\end{remark}

\section*{Acknowledgements}

We thank Professor Schulze-Pillot for his help in providing us with references about a result on random lattices. We also thank Professor Eva Bayer-Fluckiger for very useful comments and remarks on a previous version of this paper. G. Mantilla-Soler's work was supported in part by the Aalto Science Institute.

\noindent
{\footnotesize Carlos Rivera, Department of Mathematics, University of Washington,
Seattle, USA. ({\tt caariv@uw.edu})}

\noindent
{\footnotesize Guillermo Mantilla-Soler, Department of Mathematics, Universidad nacional de Colombia
Medell\'in, Colombia.  ({\tt gmantelia@gmail.com})}

\end{document}